\documentclass[11pt,reqno]{amsart}

\usepackage{caption,amsmath,amssymb,amsfonts,amsthm,latexsym,graphicx,multirow,color,hyperref,enumerate}
\usepackage[all]{xy}

\numberwithin{table}{section}

\newcommand{\A}{\mathrm{A}} \newcommand{\AGL}{\mathrm{AGL}}   \newcommand{\Aut}{\mathrm{Aut}}
 \newcommand{\bbF}{\mathbb{F}} \newcommand{\bfO}{\mathbf{O}}
     \newcommand{\calO}{\mathcal{O}}    
\newcommand{\D}{\mathrm{D}}

\newcommand{\G}{\mathrm{G}}     \newcommand{\GaU}{\mathrm{\Gamma U}}  \newcommand{\GL}{\mathrm{GL}}   \newcommand{\GU}{\mathrm{GU}}
  
\newcommand{\J}{\mathrm{J}}

\newcommand{\lefthat}{\scalebox{1.3}[1]{\text{$\hat{~}$}}}
\newcommand{\M}{\mathrm{M}} \newcommand{\magma}{\textsc{Magma}}
\newcommand{\N}{\mathrm{N}}

\newcommand{\Out}{\mathrm{Out}}
\newcommand{\Pa}{\mathrm{P}}     \newcommand{\PGaU}{\mathrm{P\Gamma U}}    \newcommand{\PSL}{\mathrm{PSL}}    \newcommand{\PSp}{\mathrm{PSp}} \newcommand{\PSU}{\mathrm{PSU}}

  \newcommand{\SL}{\mathrm{SL}}  \newcommand{\Soc}{\mathrm{Soc}} \newcommand{\Sp}{\mathrm{Sp}}  \newcommand{\SU}{\mathrm{SU}} \newcommand{\Suz}{\mathrm{Suz}} \newcommand{\Sy}{\mathrm{S}}

\newtheorem{theorem}{Theorem}[section]
\newtheorem{lemma}[theorem]{Lemma}

\newtheorem{problem}[theorem]{Problem}

\theoremstyle{definition}

\newtheorem*{remark}{Remark}

\begin{document}

\title[Factorizations of almost simple unitary groups]{Factorizations of almost simple unitary groups}

\author[Li]{Cai Heng Li}
\address{(Li) SUSTech International Center for Mathematics and Department of Mathematics\\Southern University of Science and Technology\\Shenzhen 518055\\Guangdong\\P.~R.~China}
\email{lich@sustech.edu.cn}

\author[Wang]{Lei Wang}
\address{(Wang) School of Mathematics and Statistics\\Yunnan University\\Kunming 650091\\Yunnan\\P.~R.~China}
\email{wanglei@ynu.edu.cn}

\author[Xia]{Binzhou Xia}
\address{(Xia) School of Mathematics and Statistics\\The University of Melbourne\\Parkville 3010\\VIC\\Australia}
\email{binzhoux@unimelb.edu.au}

\begin{abstract}
This is the second one in a series of papers classifying the factorizations of almost simple groups with nonsolvable factors.
In this paper we deal with almost simple unitary groups.


\end{abstract}

\maketitle

\section{Introduction}

An expression $G=HK$ of a group $G$ as the product of subgroups $H$ and $K$ is called a \emph{factorization} of $G$, where $H$ and $K$ are called \emph{factors}. A group $G$ is said to be \emph{almost simple} if $S\leqslant G\leqslant\Aut(S)$ for some nonabelian simple group $S$, where $S=\Soc(G)$ is the \emph{socle} of $G$. In this paper, by a factorization of an almost simple group we mean that none its factors contains the socle. The main aim of this paper is to solve the long-standing open problem:

\begin{problem}\label{PrbXia1}
Classify factorizations of finite almost simple groups.
\end{problem}

Determining all factorizations of almost simple groups is a fundamental problem in the theory of simple groups, which was proposed by Wielandt~\cite[6(e)]{Wielandt1979} in 1979 at The Santa Cruz Conference on Finite Groups. It also has numerous applications to other branches of mathematics such as combinatorics and number theory, and has attracted considerable attention in the literature.
In what follows, all groups are assumed to be finite if there is no special instruction.

The factorizations of almost simple groups of exceptional Lie type were classified by Hering, Liebeck and Saxl~\cite{HLS1987}\footnote{In part~(b) of Theorem~2 in~\cite{HLS1987}, $A_0$ can also be $\G_2(2)$, $\SU_3(3)\times2$, $\SL_3(4).2$ or $\SL_3(4).2^2$ besides $\G_2(2)\times2$.} in 1987.
For the other families of almost simple groups, a landmark result was achieved by Liebeck, Praeger and Saxl~\cite{LPS1990} thirty years ago, which classifies the maximal factorizations of almost simple groups. (A factorization is said to be \emph{maximal} if both the factors are maximal subgroups.)
Then factorizations of alternating and symmetric groups are classified in~\cite{LPS1990}, and factorizations of sporadic almost simple groups are classified in~\cite{Giudici2006}.
This reduces Problem~\ref{PrbXia1} to the problem on classical groups of Lie type.
Recently, factorizations of almost simple groups with a factor having at least two nonsolvable composition factors are classified in~\cite{LX2019}\footnote{In Table~1 of~\cite{LX2019}, the triple $(L,H\cap L,K\cap L)=(\Sp_{6}(4),(\Sp_2(4)\times\Sp_{2}(16)).2,\G_2(4))$ is missing, and for the first two rows $R.2$ should be $R.P$ with $P\leqslant2$.}, and those with a factor being solvable are described in~\cite{LX} and~\cite{BL}.

As usual, for a finite group $G$, we denote by $G^{(\infty)}$ the smallest normal subgroup of $X$ such that $G/G^{(\infty)}$ is solvable.
For factorizations $G=HK$ with nonsolvable factors $H$ and $K$ such that $L=\Soc(G)$ is a classical group of Lie type, the triple $(L,H^{(\infty)},K^{(\infty)})$ is described in~\cite{LWX}. Based on this work, in the present paper we characterize the triples $(G,H,K)$ such that $G=HK$ with $H$ and $K$ nonsolvable, and $G$ is a unitary group.

For groups $H,K,X,Y$, we say that $(H,K)$ contains $(X,Y)$ if $H\geqslant X$ and $K\geqslant Y$, and that $(H,K)$ \emph{tightly contains} $(X,Y)$ if in addition $H^{(\infty)}=X^{(\infty)}$ and $K^{(\infty)}=Y^{(\infty)}$. 
Our main result is the following Theorem~\ref{ThmUnitary}. 
Note that it is elementary to determine the factorizations of $G/L$ as this group has relatively simple structure (and in particular is solvable).

\begin{theorem}\label{ThmUnitary}
Let $G$ be an almost simple group with socle $L=\PSU_n(q)$, where $n\geqslant3$ and $(n,q)\neq(3,2)$, and let $H$ and $K$ be nonsolvable subgroups of $G$ not containing $L$. Then $G=HK$ if and only if (with $H$ and $K$ possibly interchanged) $G/L=(HL/L)(KL/L)$ and one of the following holds:
\begin{enumerate}[{\rm (a)}]
\item $(H,K)$ tightly contains $(X^\alpha,Y^\alpha)$ for some $(X,Y)$ in Table~$\ref{TabUnitary}$ and $\alpha\in\Aut(L)$;
\item $H^{(\infty)}=\lefthat(P.S)$ with $1<P<R$, and $S=\SL_a(q^{2b})$ ($m=ab$), $\Sp_a(q^{2b})$ ($m=ab$) or $\G_2(q^{2b})$ ($m=6b$, $q$ even), where $R=q^{m^2}$ is the unipotent radical of $\Pa_m[L]$, and $K^{(\infty)}=\SU_{m-1}(q)$ such that $HK\supseteq R$.
\end{enumerate}
\end{theorem}

\begin{remark}
Here are some remarks on Table~\ref{TabUnitary}:
\begin{enumerate}[{\rm(I)}]
\item The column $Z$ gives the smallest almost simple group with socle $L$ that contains $X$ and $Y$. In other words, $Z=\langle L,X,Y\rangle$.
It turns out that $Z=XY$ for all pairs $(X,Y)$.
\item The groups $X$, $Y$ and $Z$ are described in the corresponding lemmas whose labels are displayed in the last column. 
\item The description of groups $X$ and $Y$ are up to conjugations in $Z$ (see Lemma~\ref{LemXia04}(b) and Lemma~\ref{LemXia03}).
\end{enumerate}
\end{remark}

\begin{table}[htbp]
\captionsetup{justification=centering}
\caption{$(X,Y)$ for unitary groups}\label{TabUnitary}
\begin{tabular}{|l|l|l|l|l|l|l|}
\hline
Row & $Z$ & $X$ & $Y$ & Remarks & Lemma\\
\hline
1 & $\PSU_{2m}(q)$ & $\lefthat(q^{m^2}{:}\SL_a(q^{2b}))$ & $\SU_{2m-1}(q)$ & $m=ab$ & \ref{LemUnitary02}\\
2 & $\PSU_{2m}(q)$ & $\lefthat(q^{m^2}{:}\Sp_a(q^{2b}))$ &$\SU_{2m-1}(q)$ & $m=ab$ & \ref{LemUnitary02}\\
3 & $\PSU_{2m}(q)$ & $\lefthat(q^{m^2}{:}\G_2(q^{2b}))$ & $\SU_{2m-1}(q)$ & $m=6b$, & \ref{LemUnitary02}\\
  & & & & $q$ even & \\
\hline
4 & $\PSU_{2m}(2)$ & $(\SL_m(4).2)/d$, $\Sp_m(4).2$ & $\SU_{2m-1}(2)$ & $d=(m,3)$ & \ref{LemUnitary03}, \ref{LemUnitary05}\\
5 & $\PSU_{2m}(2).2$ & $(\SL_m(4).2)/d$, $\Sp_m(4).2$ & $\SU_{2m-1}(2).2$ & $d=(m,3)$ & \ref{LemUnitary04}, \ref{LemUnitary06}\\
6 & $\PSU_{2m}(4).4$ & $(\SL_m(16).4)/d$, $\Sp_m(16).4$ & $\SU_{2m-1}(4).4$ & $d=(m,5)$ & \ref{LemUnitary04}, \ref{LemUnitary06}\\
\hline
7 & $\PSU_{2m}(q)$ & $\PSp_{2m}(q)$ & $\SU_{2m-1}(q)$ & & \ref{LemUnitary07}\\
8 & $\PSU_6(q)$ & $\G_2(q)$ & $\SU_5(q)$ & $q$ even & \ref{LemUnitary11}\\
\hline
9 &  $\PSU_4(3).2$ & $3^4{:}(\A_5\times2)$ & $\PSL_3(4).2$ & & \ref{LemUnitary12}\\
10 &  $\PSU_4(3)$ & $\PSp_4(3)$ & $\PSL_3(4)$ & & \ref{LemUnitary13}\\
11 &  $\PSU_4(3).\calO$ & $\PSL_2(7).\calO$ & $(3^4{:}\A_6).\calO$ & $\calO\in\{4,2^2\}$ & \ref{LemUnitary14}\\
12 &  $\PSU_4(3).2$ & $\PSL_3(4).2$ & $(3^4{:}\A_6).2$ & & \ref{LemUnitary15}\\
13 &  $\PSU_4(5).2$ & $5^4{:}(\PSL_2(25)\times2)$ & $(3.\A_7)\times2$ & & \ref{LemUnitary16}\\
14 &  $\PSU_6(2)$ & $\PSU_4(3)$, $\M_{22}$ & $\SU_5(2)$ & & \ref{LemUnitary17}\\
15 &  $\PSU_9(2)$ & $\J_3$ & $2^{1+14}{:}\SU_7(2)$ & & \ref{LemUnitary18}\\
16 & $\PSU_{12}(2)$ & $\Suz$, $\G_2(4).2$ & $\SU_{11}(2)$ & & \ref{LemUnitary19}, \ref{LemUnitary20}\\
17 & $\PSU_{12}(2).2$ & $\G_2(4).2$ & $\SU_{11}(2).2$ & & \ref{LemUnitary21}\\
18 & $\PSU_{12}(4).4$ & $\G_2(16).4$ & $\SU_{11}(4).4$ & & \ref{LemUnitary21}\\
\hline
\end{tabular}
\vspace{3mm}
\end{table}

\section{Preliminaries}

In this section we collect some elementary facts regarding group factorizations.

\begin{lemma}\label{LemXia01}
Let $G$ be a group, let $H$ and $K$ be subgroups of $G$, and let $N$ be a normal subgroup of $G$. Then $G=HK$ if and only if $HK\supseteq N$ and $G/N=(HN/N)(KN/N)$.
\end{lemma}

\begin{proof}
If $G=HK$, then $HK\supseteq N$, and taking the quotient modulo $N$ we obtain
\[
G/N=(HN/N)(KN/N).
\]
Conversely, suppose that $HK\supseteq N$ and $G/N=(HN/N)(KN/N)$. Then 
\[
G=(HN)(KN)=HNK 
\]
as $N$ is normal in $G$. Since $N\subseteq HK$, it follows that $G=HNK\subseteq H(HK)K=HK$, which implies $G=HK$.
\end{proof}

Let $L$ be a nonabelian simple group. We say that $(H,K)$ is a \emph{factor pair} of $L$ if $H$ and $K$ are subgroups of $\Aut(L)$ such that $HK\supseteq L$. For an almost simple group $G$ with socle $L$ and subgroups $H$ and $K$ of $G$, Lemma~\ref{LemXia01} shows that $G=HK$ if and only if $G/L=(HL/L)(KL/L)$ and $(H,K)$ is a factor pair.
As the group $G/L$ has a simple structure (and in particular is solvable), it is elementary to determine the factorizations of $G/L$.
Thus to know all the factorizations of $G$ is to know all the factor pairs of $L$.
Note that, if $(H,K)$ is a factor pair of $L$, then any pair of subgroups of $\Aut(L)$ containing $(H,K)$ is also a factor pair of $L$.
Hence we have the following:

\begin{lemma}\label{LemXia02}
Let $G$ be an almost simple group with socle $L$, and let $H$ and $K$ be subgroups of $G$ such that $(H,K)$ contains some factor pair of $L$. Then $G=HK$ if and only if $G/L=(HL/L)(KL/L)$.
\end{lemma}

In light of Lemma~\ref{LemXia02}, the key to determine the factorizations of $G$ with nonsolvable factors is to determine the minimal ones (with respect to the containment) among factor pairs of $L$ with nonsolvable subgroups.

\begin{lemma}\label{LemXia03}
Let $L$ be a nonabelian simple group, and let $(H,K)$ be a factor pair of $L$.
Then $(H^\alpha,K^\alpha)$ and $(H^x,K^y)$ are factor pairs of $L$ for all $\alpha\in\Aut(L)$ and $x,y\in L$.
\end{lemma}

\begin{proof}
It is evident that $H^\alpha K^\alpha=(HK)^\alpha\supseteq L^\alpha=L$. Hence $(H^\alpha,K^\alpha)$ is a factor pair.
Since $xy^{-1}\in L\subseteq HK$, there exist $h\in H$ and $k\in K$ such that $xy^{-1}=hk$. Therefore, 
\[
H^xK^y=x^{-1}Hxy^{-1}Ky=x^{-1}HhkKy=x^{-1}HKy\supseteq x^{-1}Ly=L,
\]
which means that $(H^x,K^y)$ is a factor pair.
\end{proof}

The next lemma is~\cite[Lemma~2(i)]{LPS1996}.

\begin{lemma}\label{LemXia05}
Let $G$ be an almost simple group with socle $L$, and let $H$ and $K$ be subgroups of $G$ not containing $L$. If $G=HK$, then $HL\cap KL=(H\cap KL)(K\cap HL)$. 
\end{lemma}

The following lemma implies that we may consider specific representatives of a conjugacy class of subgroups when studying factorizations of a group.

\begin{lemma}\label{LemXia04}
Let $G=HK$ be a factorization. Then for all $x,y\in G$ we have $G=H^xK^y$ with $H^x\cap K^y\cong H\cap K$.
\end{lemma}
  
\begin{proof}
As $xy^{-1}\in G=HK$, there exists $h\in H$ and $k\in K$ such that $xy^{-1}=hk$. Thus
\[
H^xK^y=x^{-1}Hxy^{-1}Ky=x^{-1}HhkKy=x^{-1}HKy=x^{-1}Gy=G,
\]
and
\[
H^x\cap K^y=(H^{xy^{-1}}\cap K)^y\cong H^{xy^{-1}}\cap K=H^{hk}\cap K=H^k\cap K=(H\cap K)^k\cong H\cap K.\qedhere
\]
\end{proof}

\section{Notation}

Throughout this paper, let $q=p^f$ be a power of a prime $p$, let $n\geqslant3$ be an integer such that $(n,q)\neq(3,2)$, let $\,\overline{\phantom{\varphi}}\,$ be the homomorphism from $\GaU_n(q)$ to $\PGaU_n(q)$ modulo scalars, let $V$ be a vector space of dimension $n$ over $\bbF_{q^2}$ equipped with a nondegenerate Hermitian form $\beta$, and let $\perp$ denote the perpendicularity with respect to $\beta$.

If $n=2m$ is even, then the unitary space $V$ has a standard basis $e_1,f_1,\dots,e_m,f_m$ as in~\cite[2.2.3]{LPS1990}. In this case, let $U$ be the subspace of $V$ spanned by $e_1,\dots,e_m$, let $U_1$ be the subspace of $V$ spanned by $e_2,\dots,e_m$, let $W$ be the subspace of $V$ spanned by $f_1,\dots,f_m$, let $\phi\in\GaU(V)$ such that 
\[
\phi\colon a_1e_1+b_1f_1+\dots+a_me_m+b_mf_m\mapsto a_1^pe_1+b_1^pf_1+\dots+a_m^pe_m+b_m^pf_m
\]
for $a_1,b_1\dots,a_m,b_m\in\bbF_{q^2}$ (notice that this definition of $\phi$ is different from that in~\cite[1.7.1]{BHR2013}, see~\cite{BHR2009}), and let $\gamma$ be the involution in $\GU(V)$ swapping $e_i$ and $f_i$ for all $i\in\{1,\dots,m\}$. Then 
\[
\det(\gamma)=(-1)^m.
\]
By abuse of notation, we also let $\phi$ denote the corresponding elements in $\Aut(L)$ and $\Out(L)$.
Moreover, let $\lambda\in\bbF_{q^2}$ with $\lambda+\lambda^q=1$ (note that such $\lambda$ exists as the trace of the field extension $\bbF_{q^2}/\bbF_q$ is surjective), and let $v=e_1+\lambda f_1$. Then 
\[
\beta(v,v)=\lambda+\lambda^q=1, 
\]
and so $v$ is nonsingular. From~\cite[3.6.2]{Wilson2009} we see that $\SU(V)_U=\Pa_m[\SU(V)]$ has a subgroup $R{:}T$, where 
\[
R=q^{m^2} 
\]
is the kernel of $\SU(V)_U$ acting on $U$, and 
\[
T=\SL_m(q^2)
\] 
stabilizes both $U$ and $W$ (the action of $T$ on $U$ determines that on $W$ in the way described in~\cite[Lemma~2.2.17]{BG2016}).

\section{Infinite families of $(X,Y)$ in Table~\ref{TabUnitary}}\label{SecUnitary01}

In this subsection we construct the factor pairs $(X,Y)$ in Rows~1--8 of Table~\ref{TabUnitary}.

\begin{lemma}\label{LemUnitary01}
Let $G=\SU(V)=\SU_n(q)$ with $n=2m$, let $M=R{:}T$, and let $K=G_v$. Then the following statements hold:
\begin{enumerate}[{\rm (a)}]
\item the induced group by the action of $M\cap K$ on $U$ is $\SL(U)_{U_1,e_1+U_1}=q^{2m-2}{:}\SL_{m-1}(q^2)$;
\item the kernel of $M\cap K$ acting on $U$ is $R\cap K=q^{(m-1)^2}$;
\item $T\cap K=\SL_{m-1}(q^2)$.
\end{enumerate}
\end{lemma}

\begin{proof} 
We first calculate $R\cap K$, the kernel of $M\cap K$ acting on $U$. For each $r\in R\cap K$, since $r$ fixes $e_1$ and $v$, we deduce that $r$ fixes $\langle e_1,v\rangle_{\bbF_{q^2}}=\langle e_1,f_1\rangle_{\bbF_{q^2}}$ pointwise. 
Hence $R\cap K$ is isomorphic to the pointwise stabilizer of $U_1$ in $\SU(\langle e_2,f_2,\dots,e_m,f_m\rangle_{\bbF_{q^2}})$.
Then by~\cite[3.6.2]{Wilson2009} we have $R\cap K=q^{(m-1)^2}$.

Since $K$ fixes $v$, it stabilizes $v^\perp$. 
Hence $M\cap K$ stabilizes $U\cap v^\perp=\langle e_2,\dots,e_m\rangle_{\bbF_{q^2}}=U_1$.
For arbitrary $h\in M\cap K$, write $e_1^h=\mu e_1+e$ with $e\in U_1$. Then
\[
(\lambda f_1)^h=(v-e_1)^h=v^h-e_1^h=v-(\mu e_1+e)=(1-\mu)e_1-e+\lambda f_1,
\] 
and so $\beta(e_1^h,(\lambda f_1)^h)=\beta(\mu e_1,\lambda f_1)$. Since $\lambda\neq0$ and $h$ preserves $\beta$, we obtain $\mu=1$.
It follows that $M\cap K$ stabilizes $e_1+U_1$, and so the induced group of $M\cap K$ on $U$ is contained in $\SL(U)_{U_1,e_1+U_1}$, that is, $(M\cap K)^U\leqslant\SL(U)_{U_1,e_1+U_1}=q^{2m-2}{:}\SL_{m-1}(q^2)$. Now 
\[
M\cap K=(R\cap K).(M\cap K)^U\leqslant(R\cap K).\SL(U)_{U_1,e_1+U_1}
\] 
while
\begin{align*}
|M\cap K|\geqslant\frac{|M||K|}{|G|}&=\frac{|q^{m^2}{:}\SL_m(q^2)||\SU_{m-1}(q)|}{|\SU_m(q)|}\\
&=q^{m^2-1}|\SL_{m-1}(q^2)|=|R\cap K||\SL(U)_{U_1,e_1+U_1}|.
\end{align*}
Thus we obtain $(M\cap K)^U=\SL(U)_{U_1,e_1+U_1}=q^{2m-2}{:}\SL_{m-1}(q^2)$.

For each $t\in T\cap K$, we have $e_1^t\in U^t=U$ and $f_1^t\in W^t=W$, and then it follows from 
\[
e_1^t+\lambda f_1^t=(e_1+\lambda f_1)^t=v^t=v=e_1+\lambda f_1
\] 
that $e_1^t=e_1$ and $f_1^t=f_1$. Hence we conclude that
\[
T\cap K=T_{e_1,f_1}=T_{e_1,U_1}=\SL_{m-1}(q^2).\qedhere
\]
\end{proof}

The following lemma gives the factor pairs $(X,Y)$ in Rows 1--3 of~\ref{TabUnitary}.

\begin{lemma}\label{LemUnitary02}
Let $G=\SU(V)=\SU_n(q)$ with $n=2m$, let $H=R{:}S\leqslant G_U$ with $S=\SL_a(q^{2b})$ ($m=ab$), $\Sp_a(q^{2b})$ ($m=ab$) or $\G_2(q^{2b})$ ($m=6b$, $q$ even), let $K=G_v$, let $Z=\overline{G}$, let $X=\overline{H}$, and let $Y=\overline{K}$. Then 
\[
H\cap K=
\begin{cases}
(q^{(m-1)^2}.q^{2m-2b}){:}\SL_{a-1}(q^{2b})&\textup{if }S=\SL_a(q^{2b})\\
(q^{(m-1)^2}.[q^{2m-2b}]){:}\Sp_{a-2}(q^{2b})&\textup{if }S=\Sp_a(q^{2b})\\
(q^{(m-1)^2}.q^{4b+6b}){:}\SL_2(q^{2b})&\textup{if }S=\G_2(q^{2b}),
\end{cases}
\]
and $Z=XY$ with $Z=\PSU_n(q)$, $X=\lefthat(q^{m^2}{:}S)$ and $Y\cong K=\SU_{n-1}(q)$.
\end{lemma}

\begin{proof} 
It is clear that $Z=\PSU_n(q)$, $X=\lefthat(q^{m^2}{:}S)$, and $Y\cong K=\SU_{n-1}(q)$.
Let $M=R{:}T$. By Lemma~\ref{LemUnitary01} we have $R\cap K=q^{(m-1)^2}$ and 
\[
(M\cap K)^U=\SL(U)_{U_1,e_1+U_1}=(M_{U_1,e_1+U_1})^U.
\] 
Then $M_{U_1,e_1+U_1}=(M\cap K)R$ as $R$ is the kernel of $M$ acting on $U$.
This implies that $H_{U_1,e_1+U_1}=(H\cap K)R$, and so
\[
(H\cap K)/(R\cap K)\cong(H\cap K)R/R=H_{U_1,e_1+U_1}/R=S_{U_1,e_1+U_1}R/R\cong S_{U_1,e_1+U_1}.
\]
Hence $H\cap K=(R\cap K).S_{U_1,e_1+U_1}=q^{(m-1)^2}.S_{U_1,e_1+U_1}$. Moreover, by~\cite[Lemmas~4.2,~4.3 and~4.5]{LWX-Linear} we have
\[
S_{U_1,e_1+U_1}=
\begin{cases}
q^{2m-2b}{:}\SL_{a-1}(q^{2b})&\textup{if }S=\SL_a(q^{2b})\\
[q^{2m-2b}]{:}\Sp_{a-2}(q^{2b})&\textup{if }S=\Sp_a(q^{2b})\\
q^{4b+6b}{:}\SL_2(q^{2b})&\textup{if }S=\G_2(q^{2b}).
\end{cases}
\]
This proves the conclusion of this lemma on $H\cap K$, and implies that
\[
\frac{|H|}{|H\cap K|}=\frac{|q^{m^2}{:}S|}{|q^{(m-1)^2}.S_{U_1,e_1+U_1}|}=q^{2m-1}(q^{2m}-1)=\frac{|\SU_{2m}(q)|}{|\SU_{2m-1}(q)|}=\frac{|G|}{|K|}.
\]
Thus $G=HK$, and so $Z=\overline{G}=\overline{H}\,\overline{K}=XY$.
\end{proof}


The factor pairs $(X,Y)$ in Rows~4--6 of Table~\ref{TabUnitary} are constructed in the next four lemmas.
Note that if $q$ is even then $\gamma\in\SU(V)$ as $\det(\gamma)=1$.

\begin{lemma}\label{LemUnitary03}
Let $G=\SU(V)=\SU_n(2)$ with $n=2m$ and $q=2$, let $H=T{:}\langle\gamma\rangle$, let $K=G_v$, let $Z=\overline{G}$, let $X=\overline{H}$, and let $Y=\overline{K}$. Then $H\cap K=\SL_{m-1}(4)$, and $Z=XY$ with $Z=\PSU_n(2)$, $X=(\SL_m(4).2)/(m,3)$ and $Y\cong K=\SU_{n-1}(2)$.
\end{lemma}

\begin{proof}
Clearly, $Z=\PSU_n(2)$, $X=(\SL_m(4).2)/(m,3)$, and $Y\cong K=\SU_{n-1}(2)$. Suppose that there exists $t\in T$ with $\gamma t\in K=G_v$. Then
\[
e_1+\lambda f_1=v=v^{\gamma t}=(e_1+\lambda f_1)^{\gamma t}=(\lambda e_1+f_1)^t=\lambda e_1^t+f_1^t.
\]
Since $e_1^t\in U^t=U$ and $f_1^t\in W^t=W$, it follows that $e_1=\lambda e_1^t$ and $\lambda f_1=f_1^t$. Consequently,
\[
\lambda^q=\beta(e_1,\lambda f_1)=\beta(\lambda e_1^t,f_1^t)=\lambda\beta(e_1^t,f_1^t)=\lambda\beta(e_1,f_1)=\lambda.
\]
This implies that $\lambda+\lambda^q=2\lambda=0$, contradicting the condition $\lambda+\lambda^q=1$.

Thus we conclude that $H\cap K=T\cap K$. Then it follows from Lemma~\ref{LemUnitary01} that $H\cap K=\SL_{m-1}(4)$, and so
\[
\frac{|G|}{|K|}=\frac{|\SU_n(2)|}{|\SU_{n-1}(2)|}=2^{n-1}(2^n-1)=2\cdot4^{m-1}(4^m-1)=\frac{|\SL_m(4).2|}{|\SL_{m-1}(4)|}=\frac{|H|}{|H\cap K|}.
\]
This implies that $G=HK$, which leads to $Z=\overline{G}=\overline{H}\,\overline{K}=XY$.
\end{proof}

\begin{lemma}\label{LemUnitary04}
Let $G=\SU(V){:}\langle\phi\rangle=\SU_n(q){:}(2f)$ with $n=2m$ and $q\in\{2,4\}$, let $H=T{:}\langle\rho\rangle$ with $\rho\in\{\phi,\phi\gamma^{2/f}\}$, let $K=G_v$, let $Z=\overline{G}$, let $X=\overline{H}$, and let $Y=\overline{K}$. Then $H\cap K=\SL_{m-1}(q^2)$, and $Z=XY$ with $Z=\PSU_n(q).(2f)$, $X=(\SL_m(q^2).(2f))/(m,q+1)$ and $Y\cong K=\SU_{n-1}(q).(2f)$.
\end{lemma}

\begin{proof}
Clearly, $Z=\PSU_n(q).(2f)$, $X=(\SL_m(q^2).(2f))/(m,q+1)$, and $Y\cong K=\SU_{n-1}(q).(2f)$. Suppose $H\cap K>T\cap K$.
Then there exists $t\in T$ such that $\phi^ft\in K=G_v$. This implies that
\[
e_1+\lambda f_1=v=v^{\phi^ft}=(e_1+\lambda f_1)^{\phi^ft}=(e_1+\lambda^qf_1)^t=e_1^t+\lambda^qf_1^t.
\]
Since $e_1^t\in U^t=U$ and $f_1^t\in W^t=W$, it follows that $e_1=e_1^t$ and $\lambda f_1=\lambda^qf_1^t$. Hence
\[
\lambda^q=\beta(e_1,\lambda f_1)=\beta(e_1^t,\lambda^qf_1^t)=\lambda\beta(e_1^t,f_1^t)=\lambda\beta(e_1,f_1)=\lambda,
\]
which leads to $\lambda+\lambda^q=2\lambda=0$, contradicting the condition $\lambda+\lambda^q=1$.

Thus we conclude that $H\cap K=T\cap K$. Then it follows from Lemma~\ref{LemUnitary01} that $H\cap K=\SL_{m-1}(q^2)$. As $q\in\{2,4\}$, this implies that
\[
\frac{|G|}{|K|}=\frac{|\SU_n(q){:}(2f)|}{|\SU_{n-1}(q).(2f)|}=q^{n-1}(2^n-1)=2f\cdot q^{2m-2}(q^{2m}-1)=\frac{|\SL_m(q^2).(2f)|}{|\SL_{m-1}(q^2)|}=\frac{|H|}{|H\cap K|}.
\]
Hence $G=HK$, and so $Z=\overline{G}=\overline{H}\,\overline{K}=XY$.
\end{proof}




\begin{lemma}\label{LemUnitary05}
Let $G=\SU(V)=\SU_n(2)$ with $q=2$ and $n=2m$ for some even $m$, let $H=S{:}\langle\gamma\rangle$ with $S=\Sp_m(4)<T$, let $K=G_v$, let $Z=\overline{G}$, let $X=\overline{H}$, and let $Y=\overline{K}$. Then $H\cap K=\Sp_{m-2}(4)$, and $Z=XY$ with $Z=\PSU_n(2)$, $X=\Sp_m(4).2$ and $Y\cong K=\SU_{n-1}(2)$.
\end{lemma}

\begin{proof}
Clearly, $Z=\PSU_n(2)$, $X=\Sp_m(4).2$, and $Y\cong K=\SU_{n-1}(2)$. Since $H<T{:}\langle\gamma\rangle$, we derive from Lemmas~\ref{LemUnitary03} and~\cite[Lemma~4.6]{LWX-Linear} that
\[
H\cap K=H\cap((T{:}\langle\gamma\rangle)\cap K)=H\cap\SL_{m-1}(4)=S\cap\SL_{m-1}(4)=\Sp_{m-2}(4).
\]
Therefore,
\[
\frac{|G|}{|K|}=\frac{|\SU_n(2)|}{|\SU_{n-1}(2)|}=2^{n-1}(2^n-1)=2\cdot4^{m-1}(4^m-1)=\frac{|\Sp_m(4).2|}{|\Sp_{m-2}(4)|}=\frac{|H|}{|H\cap K|}.
\]
This implies that $G=HK$, and so $Z=\overline{G}=\overline{H}\,\overline{K}=XY$.
\end{proof}

We construct the factor pairs $(X,Y)$ in Rows~7 and~8 of Table~\ref{TabUnitary} in the following two lemmas.

\begin{lemma}\label{LemUnitary06}
Let $G=\SU(V){:}\langle\phi\rangle=\SU_n(q){:}(2f)$ with $q\in\{2,4\}$ and $n=2m$ for some even $m$, let $H=S{:}\langle\rho\rangle$ with $S=\Sp_m(q^2)<T$ and $\rho\in\{\phi,\phi\gamma^{2/f}\}$, let $K=G_v$, let $Z=\overline{G}$, let $X=\overline{H}$, and let $Y=\overline{K}$. Then $H\cap K=\Sp_{m-2}(q^2)$, and $Z=XY$ with $Z=\PSU_n(q).(2f)$, $X=\Sp_m(q^2).(2f)$ and $Y\cong K=\SU_{n-1}(q).(2f)$.
\end{lemma}

\begin{proof}
It is clear that $Z=\PSU_n(q).(2f)$, $X=\Sp_m(q^2).(2f)$, and $Y\cong K=\SU_{n-1}(q).(2f)$. By Lemmas~\ref{LemUnitary04} and~\cite[Lemma~4.6]{LWX-Linear}, we have
\[
H\cap K=H\cap((T{:}\langle\rho\rangle)\cap K)=H\cap\SL_{m-1}(q^2)=\Sp_{m-2}(q^2).
\]
As $q\in\{2,4\}$, it follows that
\[
\frac{|G|}{|K|}=\frac{|\SU_n(q){:}(2f)|}{|\SU_{n-1}(q).(2f)|}=q^{n-1}(q^n-1)=2f\cdot q^{2m-2}(q^{2m}-1)=\frac{|\Sp_m(q^2).(2f)|}{|\Sp_{m-2}(q^2)|}=\frac{|H|}{|H\cap K|},
\]
which implies $G=HK$. Thus $Z=\overline{G}=\overline{H}\,\overline{K}=XY$.
\end{proof}

\begin{lemma}\label{LemUnitary07}
Let $G=\SU(V)=\SU_n(q)$ with $n=2m$, let $H=\Sp_{2m}(q)<G$, let $K=\N_1[G]^{(\infty)}$, let $Z=\overline{G}$, let $X=\overline{H}$, and let $Y=\overline{K}$. Then $H\cap K=\Sp_{2m-2}(q)<\N_2[H]$, and $Z=XY$ with $Z=\PSU_n(q)$, $X=\PSp_{2m}(q)$ and $Y\cong K=\SU_{n-1}(q)$.
\end{lemma}

\begin{proof}
Let $\mu\in\bbF_{q^2}$ such that $\mu^{q-1}=-1$, and let $V_0=\langle\mu e_1,f_1,\dots,\mu e_m,f_m\rangle_{\bbF_q}$. 
According to~\cite[3.10.6]{Wilson2009}, there is a nondegenerate alternating form $\beta_0$ on $V_0$ such that $\mu e_1,f_1,\dots,\mu e_m,f_m$ is a standard basis of $V_0$ with respect to $\beta_0$ and $H=\Sp(V_0)$. Let $\zeta\in\bbF_{q^2}\setminus\bbF_q$, and let $u=\mu e_1+\zeta f_1$. 
Then $\beta(u,u)=\mu\zeta^q+\zeta\mu^q=\mu(\zeta^q-\zeta)\neq0$, and so we may assume without loss of generality that $K=G_u$.
Let $w=\mu e_1+\zeta^qf_1$. Then $\langle w\rangle_{\bbF_{q^2}}$ is the orthogonal complement of $\langle u\rangle_{\bbF_{q^2}}$ in the unitary space $\langle e_1,f_1\rangle_{\bbF_{q^2}}$, and
\[
u^\perp=\langle w,e_2,f_2,\dots,e_m,f_m\rangle_{\bbF_{q^2}}=\langle w,\mu e_2,f_2,\dots,\mu e_m,f_m\rangle_{\bbF_{q^2}}.
\]

Clearly, $V_0\cap u^\perp$ contains $\langle\mu e_2,f_2,\dots,\mu e_m,f_m\rangle_{\bbF_q}$. Suppose that they are not equal. 
Then there exists nonzero $\xi\in\bbF_{q^2}$ such that $\xi w\in\langle\mu e_1,f_1\rangle_{\bbF_q}$, that is, $\xi(\mu e_1+\zeta^qf_1)=a\mu e_1+bf_1$ for some $a,b\in\bbF_q$. However, this implies that $\zeta^q=b/a\in\bbF_q$, contradicting the condition $\zeta\in\bbF_{q^2}\setminus\bbF_q$.
Thus $V_0\cap u^\perp=\langle\mu e_2,f_2,\dots,\mu e_m,f_m\rangle_{\bbF_q}$. Then since $H\cap K$ stabilizes $V_0\cap u^\perp$, it stabilizes 
\[
V_1:=(\langle\mu e_2,f_2,\dots,\mu e_m,f_m\rangle_{\bbF_q})^\perp=\langle e_1,f_1\rangle_{\bbF_{q^2}}=\langle\mu e_1,f_1\rangle_{\bbF_{q^2}}.
\]
Let $V_2:=V_0\cap(V_1)^\perp=\langle\mu e_1,f_1\rangle_{\bbF_q}$. Then $H\cap K$ also stabilizes $V_2$.
For each $g\in H\cap K$, since $g|_{V_2}\in\Sp(V_2)$, we have $\det(g|_{V_2})=1$ and hence $\det(g|_{V_1})=1$.
This implies that $g|_{V_1}\in\SU(V_1)$. As $g$ fixes the nonsingular vector $u$ in the unitary space $V_1$, it then follows that $g|_{V_1}=1$.
In particular, $g$ fixes $e_1$ and $f_1$, and so $H\cap K\leqslant H_{e_1,f_1}$. Conversely, $H_{e_1,f_1}$ is obviously contained in $H_u=H\cap K$. Hence 
\[
H\cap K=H_{e_1,f_1}\cong\Sp(\langle\mu e_2,f_2,\dots,\mu e_m,f_m\rangle_{\bbF_q})=\Sp_{2m-2}(q).
\]

Now we have
\[
\frac{|G|}{|K|}=\frac{|\SU_{2m}(q)|}{|\SU_{2m-1}(q)|}=q^{2m-1}(q^{2m}-1)=\frac{|\Sp_{2m}(q)|}{|\Sp_{2m-2}(q)|}=\frac{|H|}{|H\cap K|}.
\]
Thus $G=HK$, and so $Z=\overline{G}=\overline{H}\,\overline{K}=XY$.
\end{proof}

\begin{lemma}\label{LemUnitary11}
Let $G=\SU(V)=\SU_6(q)$ with $n=6$ and $q$ even, let $H=\G_2(q)<\Sp_6(q)<G$, let $K=G_v$, let $Z=\overline{G}$, let $X=\overline{H}$, and let $Y=\overline{K}$. 
Then $H\cap K=\SL_2(q)$, and $Z=XY$ with $Z=\PSU_6(q)$, $X=\G_2(q)$ and $Y\cong K=\SU_5(q)$.
\end{lemma}

\begin{proof}
It is clear that $Z=\PSU_6(q)$, $X=\G_2(q)$, and $Y\cong K=\SU_5(q)$. By Lemmas~\ref{LemUnitary07} and~\cite[Lemma~4.11]{LWX-Linear} we have
\[
H\cap K=H\cap(\Sp_6(q)\cap K)=H\cap\Sp_4(q)=\SL_2(q).
\]
Hence
\[
\frac{|G|}{|K|}=\frac{|\SU_6(q)|}{|\SU_5(q)|}=q^5(q^6-1)=\frac{|\G_2(q)|}{|\SL_2(q)|}=\frac{|H|}{|H\cap K|},
\]
and so $G=HK$, which implies that $Z=\overline{G}=\overline{H}\,\overline{K}=XY$.
\end{proof}

\begin{remark}
If we let $H=\G_2(q)'$ in Lemma~\ref{LemUnitary11} then the conclusion $Z=XY$ would not hold for $q=2$. 
\end{remark}

\section{Sporadic cases of $(X,Y)$ in Table~\ref{TabUnitary}}\label{SecUnitary02}

The factor pairs $(X,Y)$ in Rows~9--14 of Table~\ref{TabUnitary} are constructed in Lemmas~\ref{LemUnitary12}--\ref{LemUnitary17} below, which are verified by computation in \magma~\cite{BCP1997}. The maximal subgroups of almost simple groups with socle $\PSU_4(3)$ can be found in ~\cite{CCNPW1985}.

\begin{lemma}\label{LemUnitary12}
Let $L=\PSU_4(3)$, let $Z=L.2$ be an almost simple group with socle $L$ such that $Z$ has a maximal subgroup $A=3^4{:}(\A_6\times2)$, let $X=3^4{:}(\A_5\times2)$ be a subgroup of $A$ (there are two conjugacy classes of such subgroups $X$), and let $Y=\PSL_3(4).2$ be a maximal subgroup of $G$. Then $Z=XY$ with $X\cap Y=\A_5$.
\end{lemma}

\begin{lemma}\label{LemUnitary13}
Let $Z=\PSU_4(3)$, let $X=\PSp_4(3)<Z$ (there are two conjugacy classes of such subgroups $X$), and let $Y=\PSL_3(4)<Z$ (there are two conjugacy classes of such subgroups $Y$). Then $Z=XY$ with $X\cap Y=2^4{:}\D_{10}$.
\end{lemma}

\begin{lemma}\label{LemUnitary14}
Let $L=\PSU_4(3)$, let $Z$ be an almost simple group with socle $L$ such that $Z/L=4$ or $2^2$ and $Z$ has no maximal subgroup of the form $\PSU_4(2).2^2$, let $X=\PSL_2(7).(G/L)<Z$ (there is a unique conjugacy class of such subgroups $X$), and let $Y=\Pa_2[Z]$. Then $Z=XY$ with $X\cap Y=\Sy_3$.
\end{lemma}

\begin{lemma}\label{LemUnitary15}
Let $L=\PSU_4(3)$, let $Z=L.2$ be an almost simple group with socle $L$ such that $Z$ has a maximal subgroup of the form $\PSL_3(4).2$, and let $Y=\Pa_2[Z]$.
Then $Y=3^4{:}(\A_6\times2)$ or $3^4{:}\M_{10}$. 
If $Y=3^4{:}(\A_6\times2)$, then each maximal subgroup $X$ of $Z$ of the form $\PSL_3(4).2$ satisfies $Z=XY$. 
If $Y=3^4{:}\M_{10}$, then there is precisely one conjugacy class of maximal subgroups $X$ of $Z$ of the form $\PSL_3(4).2$ such that $Z=XY$. 
For each such pair $(X,Y)$ we have $X\cap Y=\A_6$.
\end{lemma}

\begin{lemma}\label{LemUnitary16}
Let $L=\PSU_4(5)$, let $Z=L.2$ be an almost simple group with socle $L$ such that $Z$ has no maximal subgroup of the form $\A_7.2$ (refer to~\cite[Table~8.11]{BHR2013}), let $X=\Pa_2[Z]$, and let $Y=(3.\A_7).2<Z$ (there is a unique conjugacy class of such subgroups $Y$). Then $Z=XY$ with $X\cap Y=\AGL_1(5)$.
\end{lemma}

\begin{lemma}\label{LemUnitary17}
Let $Z=\PSU_6(2)$, let $X$ be a subgroup of $Z$ isomorphic to $\PSU_4(3)$ or $\M_{22}$ (there are precisely three conjugacy classes of such subgroups $X$ in each case), and let $Y=\N_1[Z]^{(\infty)}=\SU_5(2)$. Then $Z=XY$ with 
\[
X\cap Y=
\begin{cases}
3^4{:}\A_5&\textup{if }X=\PSU_4(3)\\
\PSL_2(11)&\textup{if }X=\M_{22}.
\end{cases}
\]
\end{lemma}

The next two lemmas are proved in~\cite[5.2.12]{LPS1990} and~\cite[4.4.2]{LPS1990}, respectively.

\begin{lemma}\label{LemUnitary18}
Let $Z=\PSU_9(2)$, let $X=\J_3$ be a maximal subgroup of $Z$ (there are precisely three conjugacy classes of such subgroups $X$, see~\cite[Table~8.57]{BHR2013}), and let $Y=\Pa_1[Z]=2^{1+14}{:}\SU_7(2)$. Then $Z=XY$ with $X\cap Y=2^{2+4}.(3\times\Sy_3)$.
\end{lemma}

\begin{lemma}\label{LemUnitary19}
Let $Z=\PSU_{12}(2)$, let $X=\Suz$ be a maximal subgroup of $Z$ (there are precisely three conjugacy classes of such subgroups $X$, see~\cite[Table~8.79]{BHR2013}), and let $Y=\N_1[Z]=\SU_{11}(2)$. Then $Z=XY$ with $X\cap Y=3^5.\PSL_2(11)$.
\end{lemma}

In the following two lemmas we construct the factor pairs $(X,Y)$ in Rows~16--18 of Table~\ref{TabUnitary}.

\begin{lemma}\label{LemUnitary20}
Let $G=\SU(V)=\SU_{12}(2)$ with $n=12$ and $q=2$, let $H=S{:}\langle\gamma\rangle$ with $S=\G_2(4)<T$, let $K=G_v$, let $Z=\overline{G}$, let $X=\overline{H}$, and let $Y=\overline{K}$. Then $H\cap K=\SL_2(4)$, and $Z=XY$ with $Z=\PSU_{12}(2)$, $X=\G_2(4).2$ and $Y\cong K=\SU_{11}(2)$.
\end{lemma}

\begin{proof}
Clearly, $Z=\PSU_{12}(2)$, $X=\G_2(4).2$, and $Y\cong K=\SU_{11}(2)$. Since $H<T{:}\langle\gamma\rangle$, we derive from Lemmas~\ref{LemUnitary03} and~\cite[Lemma~4.12]{LWX-Linear} that
\[
H\cap K=H\cap((T{:}\langle\gamma\rangle)\cap K)=H\cap\SL_5(4)=S\cap\SL_5(4)=\SL_2(4).
\]
Therefore,
\[
\frac{|G|}{|K|}=\frac{|\SU_{12}(2)|}{|\SU_{11}(2)|}=2^{11}(2^{12}-1)=2\cdot4^5(4^6-1)=\frac{|\G_2(4).2|}{|\SL_2(4)|}=\frac{|H|}{|H\cap K|}.
\]
This implies that $G=HK$, and so $Z=\overline{G}=\overline{H}\,\overline{K}=XY$.
\end{proof}

\begin{lemma}\label{LemUnitary21}
Let $G=\SU(V){:}\langle\phi\rangle=\SU_{12}(q){:}(2f)$ with $n=12$ and $q\in\{2,4\}$, let $H=S{:}\langle\rho\rangle$ with $S=\G_2(q^2)<T$ and $\rho\in\{\phi,\phi\gamma^{2/f}\}$, let $K=G_v$, let $Z=\overline{G}$, let $X=\overline{H}$, and let $Y=\overline{K}$. Then $H\cap K=\SL_2(q^2)$, and $Z=XY$ with $Z=\PSU_{12}(q).(2f)$, $X=\G_2(q^2).(2f)$ and $Y\cong K=\SU_{11}(q).(2f)$.
\end{lemma}

\begin{proof}
It is clear that $Z=\PSU_{12}(q).(2f)$, $X=\G_2(q^2).(2f)$, and $Y\cong K=\SU_{11}(q).(2f)$. By Lemmas~\ref{LemUnitary04} and~\cite[Lemma~4.12]{LWX-Linear}, we have
\[
H\cap K=H\cap((T{:}\langle\rho\rangle)\cap K)=H\cap\SL_5(q^2)=\SL_2(q^2).
\]
As $q\in\{2,4\}$, it follows that
\[
\frac{|G|}{|K|}=\frac{|\SU_{12}(q){:}(2f)|}{|\SU_{11}(q).(2f)|}=q^{11}(q^{12}-1)=2f\cdot q^{10}(q^{12}-1)=\frac{|\G_2(q^2).(2f)|}{|\SL_2(q^2)|}=\frac{|H|}{|H\cap K|},
\]
which implies $G=HK$. Thus $Z=\overline{G}=\overline{H}\,\overline{K}=XY$.
\end{proof}

\section{Proof of Theorem~\ref{ThmUnitary}}

Let $G$ be an almost simple group with socle $L=\PSU_n(q)$, and let $H$ and $K$ be nonsolvable subgroups of $G$ not containing $L$.
In Subsections~\ref{SecUnitary01} and~\ref{SecUnitary02} it is shown that all pairs $(X,Y)$ in Table~\ref{TabUnitary} are factor pairs of $L$.
Hence Lemma~\ref{LemXia02} asserts that, if $(H,K)$ contains any of these pairs $(X,Y)$ and $G/L=(HL/L)(KL/L)$, then $G=HK$.
Conversely, if $G=HK$, then by~\cite[Theorem~4.1]{LWX} the triple $(L,H^{(\infty)},K^{(\infty)})$ lies in Table~\ref{TabInftyUnitary}.

\begin{table}[htbp]
\captionsetup{justification=centering}
\caption{$(L,H^{(\infty)},K^{(\infty)})$ for unitary groups}\label{TabInftyUnitary}
\begin{tabular}{|l|l|l|l|l|l|}
\hline
Row & $L$ & $H^{(\infty)}$ & $K^{(\infty)}$ & Conditions\\
\hline
1 & $\PSU_{2m}(q)$ & $\lefthat(P.\SL_a(q^{2b}))$ ($m=ab$), & $\SU_{2m-1}(q)$ & $P\leqslant q^{m^2}$\\
 & & $\lefthat(P.\Sp_a(q^{2b}))$ ($m=ab$), & & \\
 & & $\lefthat(P.\G_2(q^{2b}))$ ($m=6b$, $q$ even), & & \\
 & & $\lefthat\Sp_{2m}(q)$ & & \\
\hline
2 & $\PSU_6(q)$ & $\G_2(q)'$ & $\SU_5(q)$ & $q$ even\\
\hline
3 & $\PSU_4(3)$ & $3^4{:}\A_5$, $\PSp_4(3)$ & $\PSL_3(4)$ & \\
4 & $\PSU_4(3)$ & $\PSL_2(7)$, $\PSL_3(4)$ & $3^4{:}\PSL_2(9)$ & \\
5 & $\PSU_4(5)$ & $5^4{:}\PSL_2(25)$ & $3.\A_7$ & \\
6 & $\PSU_6(2)$ & $\PSU_4(3)$, $\M_{22}$ & $\SU_5(2)$ & \\
7 & $\PSU_9(2)$ & $\J_3$ & $2^{1+14}{:}\SU_7(2)$ & \\
8 & $\PSU_{12}(2)$ & $\Suz$ & $\SU_{11}(2)$ & \\
\hline
\end{tabular}
\vspace{3mm}
\end{table}

For $(L,H^{(\infty)},K^{(\infty)})$ in Row~2 of Table~\ref{TabInftyUnitary}, viewing the remark after Lemma~\ref{LemUnitary11}, we see that if $G=HK$ then $(H,K)$ tightly contains the pair $(X,Y)=(\G_2(q),\SU_5(q))$ in Row~8 of Table~\ref{TabUnitary}.

For $(L,H^{(\infty)},K^{(\infty)})$ in Rows~3--6 of Table~\ref{TabInftyUnitary}, computation in \magma~\cite{BCP1997} shows that if $G=HK$ then $(H,K)$ tightly contains some pair $(X,Y)$ in Rows~9--14 of Table~\ref{TabUnitary}.

If $(L,H^{(\infty)},K^{(\infty)})$ lies in Rows~7 and~8 of Table~\ref{TabInftyUnitary}, then the pair $(H,K)$ tightly contains $(X,Y)=(H^{(\infty)},K^{(\infty)})$ in Rows~15 and~16, respectively, of Table~\ref{TabUnitary}.

Now assume that $(L,H^{(\infty)},K^{(\infty)})$ lies in Row~1 of Table~\ref{TabInftyUnitary}. Then $n=2m$ and $K^{(\infty)}=\SU_{2m-1}(q)$.
If $H^{(\infty)}=\lefthat\Sp_{2m}(q)$, then $(H,K)$ tightly contains the pair $(X,Y)=(\PSp_{2m}(q),\SU_{2m-1}(q))$ in Row~7 of Table~\ref{TabUnitary}.
Thus assume that 
\[
H^{(\infty)}=\lefthat(P.S)
\] 
with $P\leqslant q^{m^2}$ and $S=\SL_a(q^{2b})$ ($m=ab$), $\Sp_a(q^{2b})$ ($m=ab$) or $\G_2(q^{2b})$ ($m=6b$, $q$ even).
If $P=q^{m^2}$, then $(H,K)$ tightly contains the pair $(X,Y)=(\lefthat(R{:}S),\SU_{2m-1}(q))$ in Rows~1--3 of Table~\ref{TabUnitary}.
If $1<P<q^{m^2}$, then the next lemma shows that $G=HK$ if and only if part~(b) of Theorem~\ref{ThmUnitary} holds.

\begin{lemma}
Let $(H^{(\infty)},K^{(\infty)})=(\lefthat(P.S),\SU_{2m-1}(q))$ be as above such that $1<P<q^{m^2}$. Then $G=HK$ if and only if $G/L=(HL/L)(KL/L)$ and $HK\supseteq R$.
\end{lemma}

\begin{proof}
Obviously, if $G=HK$ then $G/L=(HL/L)(KL/L)$ and $HK\supseteq R$. Suppose conversely that $G/L=(HL/L)(KL/L)$ and $HK\supseteq R$.
Let $M=\overline{R{:}T}\cong R.T$. Then by Lemma~\ref{LemUnitary02} we have $MK\supseteq MK^{(\infty)}\supseteq L$ with 
\[
(M\cap K)R/R\geqslant q^{2m-2}{:}\SL_{m-1}(q^2).
\]
Since $S\lesssim HR/R$, we then conclude from~\cite[Lemmas~4.2,~4.3 and~4.5]{LWX-Linear} that 
\[
T=\SL_m(q)\subseteq(HR/R)((M\cap K)R/R).
\]
This together with the condition $HK\supseteq R$ implies that $HK\supseteq R.T=M$ and so $HK\supseteq MK\supseteq L$,
Thus by Lemma~\ref{LemXia01} we obtain $G=HK$ as $G/L=(HL/L)(KL/L)$.
\end{proof}

Finally, if $P=1$ then the following lemma shows that $(H,K)$ tightly contains $(X^\alpha,Y^\alpha)$ for some pair $(X,Y)$ in Rows~4--6 of Table~\ref{TabUnitary} and $\alpha\in\Aut(L)$.

\begin{lemma}
Suppose that $G=HK$ with $H^{(\infty)}=\lefthat S$ and $K^{(\infty)}=\SU_{2m-1}(q)$, where $S=\SL_a(q^{2b})$ ($m=ab$), $\Sp_a(q^{2b})$ ($m=ab$) or $\G_2(q^{2b})$ ($m=6b$, $q$ even). Then $(H,K)$ tightly contains $(X^\alpha,Y^\alpha)$ for some pair $(X,Y)$ in Rows~\emph{4--6} or~\emph{16--18} of Table~$\ref{TabUnitary}$ and $\alpha\in\Aut(L)$.
\end{lemma}

\begin{proof}
Since $\bfO_p(H)=1$, we see from the classification of maximal factorizations of $G$ (see~\cite[Theorem~A]{LPS1990} and~\cite{LPS1996}) that $H$ is contained in a maximal subgroup $A$ of $G$ such that either $A\cap L=\lefthat\SL_m(q^2).(q-1).2$ with $q\in\{2,4\}$ or $A^{(\infty)}=\PSp_{2m}(q)$. 
Suppose that $A^{(\infty)}=\PSp_{2m}(q)$. Then $H$ is contained in a quadratic field extension subgroup of $A$, that is, $H$ stabilizes a vector space of dimension $m$ over $\bbF_{q^2}$. Since $|H|$ is divisible by $q^{2m}-1$, we then conclude that $H$ stabilizes a totally isotropic $m$-subspace of $V$.
This together with $\bfO_p(H)=1$ implies that $H$ is contained in a maximal subgroup $A$ of $G$ such that either $A\cap L=\lefthat\SL_m(q^2).(q-1).2$.

Thus there always exists a maximal subgroup $A$ of $G$ containing $H$ such that $A\cap L=\lefthat\SL_m(q^2).(q-1).2$ with $q\in\{2,4\}$. 
Moreover,~\cite[Theorem~A]{LPS1990} shows that $G\geqslant L.4$ if $q=4$. Note that applying this conclusion to the factorization $HL\cap KL=(H\cap KL)(K\cap HL)$ gives $HL/L\geqslant L.4$ and $KL/L\geqslant L.4$ for the case $q=4$.
Hence, up to a conjugation of some $\alpha\in\Aut(L)$ on $H$ and $K$ at the same time, one of the following appears:
\begin{enumerate}[{\rm (i)}]
\item $q=2$, and $H\leqslant\overline{D{:}\langle\phi\gamma\rangle}$ with $D=\GL_m(q^2)$ stabilizing $U$ and $W$ respectively;
\item $q=2$, and $H\geqslant\overline{S{:}\langle\phi\rangle}$ or $\overline{S{:}\langle\gamma\rangle}$;
\item $q=4$, and $H\geqslant\overline{S{:}\langle\phi\rangle}$ or $\overline{S{:}\langle\phi\gamma\rangle}$.
\end{enumerate}

Suppose that~(i) holds. Let $C=\overline{D{:}\langle\phi\gamma\rangle}$, and let $B$ be the subgroup of $G$ stabilizing $\langle v\rangle_{\bbF_{q^2}}$.
Without loss of generality, assume that $K\leqslant B$. Then we have $G=CB$.
Since $\beta(\lambda^qe_1,\lambda^{-1}f_1)=\beta(e_1,f_1)$, there exists $t\in D$ such that $e_1^t=\lambda^qe_1$ and $f_1^t=\lambda^{-1}f_1$. 
It follows that 
\[
v^{t\phi\gamma}=(e_1+\lambda f_1)^{t\phi\gamma}=(\lambda^qe_1+f_1)^{\phi\gamma}=(\lambda e_1+f_1)^\gamma=e_1+\lambda f_1=v,
\]
which means $\overline{t\phi\gamma}\in B$. Since Lemma~\ref{LemUnitary01} implies $\overline{D}\cap B\geqslant\SL_{m-1}(q^2)$, we then derive that $C\cap B\geqslant\SL_{m-1}(q^2).2$. Consequently, 
\[
\frac{|C|_2}{|C\cap B|_2}\leqslant\frac{|\SL_m(q^2).2|_2}{|\SL_{m-1}(q^2).2|_2}=q^{2m-2}<q^{2m-1}=\frac{|\SU_{2m}(q)|_2}{|\SU_{2m-1}(q)|_2}\leqslant\frac{|G|_2}{|B|_2},
\]
contradicting the factorization $G=CB$.

Thus we conclude that one of~(ii) or~(iii) appears. Note that this conclusion also holds for the factorization $HL\cap KL=(H\cap KL)(K\cap HL)$.
Hence $(H,K)$ tightly contains some pair $(X,Y)$ in Rows~4--6 or~16--18 of Table~\ref{TabUnitary}.
\end{proof}

\section*{Acknowledgments}
The first author acknowledges the support of NNSFC grants no.~11771200 and no.~11931005. The second author acknowledges the support of NNSFC grant no.~12061083.


\begin{thebibliography}{}

\bibitem{BCP1997}
W. Bosma, J. Cannon and C. Playoust,
The magma algebra system I: The user language,
\emph{J. Symbolic Comput.}, 24 (1997), no. 3-4, 235--265.

\bibitem{BHR2009}
J. N. Bray, D. F. Holt and C. M. Roney-Dougal,
Certain classical groups are not well-defined,
\emph{J. Group Theory}, 12 (2009), 171--180.

\bibitem{BHR2013}
J. N. Bray, D. F. Holt and C. M. Roney-Dougal,
\emph{The maximal subgroups of the low-dimensional finite classical groups}, Cambridge University Press, Cambridge, 2013.

\bibitem{BG2016}
T. Burness and M. Giudici, 
\emph{Classical Groups, Derangements and Primes}, Cambridge University Press, Cambridge, 2016.

\bibitem{BL}
T. C. Burness and C. H. Li,
On solvable factors of almost simple groups,
\emph{Adv. Math.}, 377 (2021), 107499, 36 pp.


\bibitem{CK1979}
P. J. Cameron and W. M. Kantor,
$2$-transitive and antiflag transitive collineation groups of finite projective spaces,
\emph{J. Algebra}, 60 (1979), no. 2, 384--422.


\bibitem{CCNPW1985}
J. H. Conway, R. T. Curtis, S. P. Norton, R. A. Parker and R. A. Wilson,
\emph{Atlas of finite groups: maximal subgroups and ordinary characters for simple groups}, Clarendon Press, Oxford, 1985.








\bibitem{Giudici2006}
M. Giudici,
Factorisations of sporadic simple groups,
\emph{J. Algebra}, 304 (2006), no. 1, 311--323.

\bibitem{GGP}
M. Giudici, S. P. Glasby and C. E. Praeger,
Subgroups of Classical Groups that are Transitive on Subspaces,
https://arxiv.org/abs/2012.07213.






\bibitem{HLS1987}
C. Hering, M. W. Liebeck and J. Saxl,
The factorizations of the finite exceptional groups of Lie type,
\emph{J. Algebra}, 106 (1987), no. 2, 517--527.





\bibitem{Kantor}
W. M. Kantor,
Antiflag transitive collineation groups,
https://arxiv.org/pdf/1806.02203.








\bibitem{LWX}
C. H. Li, L. Wang and B. Xia,
The exact factorizations of almost simple groups,
https://arxiv.org/abs/2012.09551.

\bibitem{LWX-Linear}
C. H. Li, L. Wang and B. Xia,
Factorizations of almost simple linear groups,
https://arxiv.org/abs/2106.03109.

\bibitem{LX2019}
C. H. Li and B. Xia,
Factorizations of almost simple groups with a factor having many nonsolvable composition factors,
\emph{J. Algebra}, 528 (2019), 439--473.

\bibitem{LX}
C. H. Li and B. Xia,
Factorizations of almost simple groups with a solvable factor, and Cayley graphs of solvable groups,
to appear in \emph{Mem. Amer. Math. Soc.}, https://arxiv.org/abs/1408.0350.

\bibitem{Liebeck1987}
M. W. Liebeck,
The affine permutation groups of rank three,
\emph{Proc. London Math. Soc. (3)}, 54 (1987), no. 3, 477--516.

\bibitem{LPS1990}
M. Liebeck, C. E. Praeger and J. Saxl,
The maximal factorizations of the finite simple groups and their automorphism groups,
\emph{Mem. Amer. Math. Soc.},  86  (1990),  no. 432.

\bibitem{LPS1996}
M. W. Liebeck, C. E. Praeger and J. Saxl, On factorizations of almost simple groups,
\emph{J. Algebra}, 185 (1996), no. 2, 409--419.

\bibitem{LPS2000}
M. W. Liebeck, C. E. Praeger and J. Saxl, Transitive subgroups of primitive permutation groups,
\emph{J. Algebra}, 234 (2000), no. 2, 291--361.

\bibitem{LPS2010}
M. W. Liebeck, C. E. Praeger and J. Saxl, Regular subgroups of primitive permutation groups,
\emph{Mem. Amer. Math. Soc.}, 203 (2010),  no. 952.



\bibitem{Wielandt1979}
H. Wielandt,
Zusammengesetzte Gruppen: H\"{o}lders Programm heute,
in \emph{The Santa Cruz Conference on Finite Groups (Univ. California, Santa Cruz, Calif., 1979)}, pp. 161--173.



\bibitem{Wilson2009}
R. Wilson,
\emph{The finite simple groups}, Springer, 2009.


\bibliographystyle{100}

\end{thebibliography}
\end{document}